\documentclass[preprint,12pt]{elsarticle}
\usepackage{amssymb,amsthm,amsmath}

\newtheorem{exa}{Example}
\newtheorem{theo}{Theorem}

\newtheorem{prop}{Proposition}

\newtheorem{lemma}{Lemma}

\newtheorem{remark}{Remark}

\journal{some journal}

\begin{document}

\begin{frontmatter}

\title{A Berry Esseen type limit theorem for Boolean convolution}
\author[aut1]{Octavio Arizmendi}
\ead{octavius@cimat.mx}
\author[aut1]{Mauricio Salazar}
\ead{maurma@cimat.mx}
\address[aut1]{Centro de Investigaci\'on en Matem\'aticas.}

\begin{abstract}
We give estimates on the rate of convergence in the Boolean central limit theorem for the  L\'evy distance. In the case of measures with bounded support we obtain a sharp estimate by giving a qualitative description of this convergence.

\end{abstract}


\end{frontmatter}

\section{Introduction}

In Non Commutative Probability, as proved by Muraki \cite{Mur2}, there are essentially four notions of independence: classical, free, Boolean, and monotone.  For each type of independence there exists a Central Limit Theorem stating that the normalized sum of independent random variables with finite variance converges to the Gaussian, semicircle \cite{Voi}, Bernoulli \cite{SW} and arcsine \cite{Mur} distributions, respectively .

In applications, in order to apply effectively any limit theorem, one needs a quantitative version of it. For the central limit theorem in probability, this is known as the Berry-Esseen Theorem \cite{Ber,Es}. It states that if $\mu$ is a probability measure with $m_1(\mu)=0$, $m_2(\mu)=1$ and $\int_{\mathbb{R}} | x |^3 d\mu < \infty$, then the distance to the standard Gaussian distribution $\mathcal{N}$ is bounded as follows
$$d_{kol}(D_{\frac{1}{\sqrt{n}}}\mu^{\ast n},\mathcal{N})\leq C \frac{\int_{\mathbb{R}} | x |^3 d\mu}{\sqrt{n}},$$
where $d_{kol}$ denotes the Kolmogorov distance between measures, $D_b\mu$ denotes the dilation of a measure $\mu$ by a factor $b>0$,  $\ast$ denotes the classical convolution, and $C$ is an absolute constant.

For the Free Central Limit Theorem, a similar result was given by Kargin for the bounded case and then broadly improved by Chistyakov and G\"{o}tze \cite{CG} using both the third and fourth moment; if $\mu$ is a probability measure with $m_1(\mu)=0$, $m_2(\mu)=1$ and $m_4(\mu)< \infty$, then
the distance to the standard semicircle distribution $\mathcal{S}$ satisfies 
$$d_{kol}(D_{\frac{1}{\sqrt{n}}}\mu^{\boxplus n},\mathcal{S})\leq C' \frac{|m_3(\mu)|+|m_4(\mu)|^{1/2}}{\sqrt{n}},$$
where the symbol $\boxplus$ denotes the free convolution, and $C'$ is an absolute constant.

In this paper we study the speed of convergence in the Central limit for Boolean convolution. Both of the above results are given in terms of the Kolmogorov distance.  However, in the Boolean Central Limit Theorem there is not convergence in the Kolmogorov distance, as one can easily see from almost any example (see Example \ref{ejemplo} for a particular one). Thus, in this paper we consider the L\'evy distance instead, which seems the most appropriate.

Our first theorem gives a quantitative version of the Boolean Central Limit Theorem for measures with finite fourth moment. That is, we give an estimate for the L\'evy distance to the Bernoulli distribution $\mathbf{b}=\frac{1}{2}\delta_{-1}+\frac{1}{2}\delta_1$.
\begin{theo}\label{teo.noacotado} 
Let $\mu$ be a probability measure such that $m_1(\mu)=0$, $m_2(\mu)=1$, and $m_4(\mu)<\infty$. Then for the measure $\mu_n=D_{\frac{1}{\sqrt{n}}} \mu^{\uplus n}$ we have that 

 $$L(\mu_n,\mathbf{b})\leq \frac{7}{2}\sqrt[3]{\frac{m_4(\mu)-1}{n}}\: \text{for } n\ge 1,$$
where $\uplus$ denotes the Boolean convolution.
\end{theo}

Our second contribution specializes in the case of measures with bounded support.  In this case we give a qualitative description of the Boolean Central Limit Theorem which allows us to conclude a better bound for the L\'evy distance to the Bernoulli distribution.
\begin{theo}\label{teo.acotado}
Let $\mu$ be a probability measure such that $m_1(\mu)=0$, $m_2(\mu)=1$, and $supp(\mu)\in [-K,K]$. Then the measure   $\mu_n:=D_{\frac{1}{\sqrt{n}}} \mu^{\uplus n}$ satisfies for $\sqrt{n}>K$ that:

\begin{enumerate}
\item[1)] $supp\:\mu_{n}\subset [\frac{-K}{\sqrt{n}},\frac{K}{\sqrt{n}}]\cup \lbrace x_1,x_2 \rbrace$, where $|(-1)-x_1|\leq \frac{K}{\sqrt{n}}$ and $|1-x_2|\leq \frac{K}{\sqrt{n}}$.
\item[2)] For $p=\mu_n(\lbrace x_1 \rbrace)$, $q=\mu_n(\lbrace x_2 \rbrace)$ and $r=\mu_n([\frac{-K}{\sqrt{n}},\frac{K}{\sqrt{n}}])$, we have that $p,q\in [\frac{1}{2}-\frac{2K}{\sqrt{n}},\frac{1}{2}+\frac{K}{2\sqrt{n}}]$ and $r<\frac{4K}{\sqrt{n}}$.
\end{enumerate} 
In particular, the L\'evy distance between $\mu_n$ and $\mathbf{b}$ is bounded by
\begin{equation*}
L(\mu_{n},\mathbf{b})\le \frac{2K}{\sqrt{n}}.
\end{equation*}
\end{theo}

The remainder of the paper is organized as follows: In Section 2 we recall the definition of Boolean convolution in an analytic way. For that purpose we introduce the Cauchy transform, the $F$-transform and the $K$-transform. Then we study some properties of the $F$-transform that are used to prove Theorem \ref{teo.noacotado}. Finally we introduce the Boolean cumulants; they are the base of the proof of Theorem \ref{teo.acotado}. In Section 3 we prove Theorem \ref{teo.noacotado} and Theorem \ref{teo.acotado}, and by an example we show that the estimate in the latter is sharp.

\section{Preliminaries}

We denote by $\mathcal{M}$ the set of Borel probability measures on $\mathbb{R}$ and by $\mathcal{M}_0^1$ the subset of $\mathcal{M}$ of probability measures with zero mean and unit variance.  The support of a measure $\mu \in \mathcal{M}$ is denoted by $supp(\mu)$. For $\mu\in \mathcal{M}$ let $D_a\mu$ denote the dilation of a measure $\mu$ by a factor $a>0$; this means that $D_a\mu(B)=\mu(a^{-1}B)$ for all Borel sets $B \subset \mathbb{R}$.

For $\mu,\nu \in \mathcal{M}$ define the L\'evy distance between them to be
$$L(\mu,\nu):=inf \{ \epsilon>0 \: | \: F(x-\epsilon)-\epsilon\leq G(x)\leq F(x+\epsilon)+\epsilon \: \: \: \text{for all } x\in \mathbb{R} \},$$
where $F$ and $G$ are the cumulative distribution functions of $\mu$ and $\nu$ respectively.

It is well known that the L\'evy distance of two measures is the side-length of the largest square that can be inscribed between the graphs, adding vertical segments where there is a discontinuity, of the cumulative distribution functions of the measures.

\subsection{Boolean convolution}

Boolean convolution is the main object of study in this paper.  It corresponds to the sum of Boolean-independent random variables. We introduce it in a purely analytic way, so we skip the framework of non-commutative probability spaces.

By $\mathbb{C}^+$  and $\mathbb{C}^-$ we denote the open upper and lower complex half-planes, respectively.
The Cauchy transform of a measure $\mu \in \mathcal{M}$ is defined as
$$G_{\mu}(z):=\int_{\mathbb{R}}\frac{1}{z-t}d\mu(t)\quad \text{for}\:z\in \mathbb{C}^+.$$
In fact the Cauchy transform is well defined in $\mathbb{C} \setminus supp(\mu)$, but it is defined only for $z\in \mathbb{C}^+$ because in that region the Cauchy transform determines uniquely to the measure (i.e. $G_{\mu}(z)=G_{\nu}(z)$ for $z\in \mathbb{C}^+$ implies $\mu=\nu$ ). More precisely, we can recover a measure $\mu \in \mathcal{M}$ from its Cauchy transform via the Stieltjes inversion formula:
\begin{equation}
\mu((a,b])=\displaystyle\lim_{\delta \downarrow 0}\displaystyle\lim_{\epsilon \downarrow 0}-\frac{1}{\pi}\int_{a+\delta}^{b+\delta}Im(G_{\mu}(x+i\epsilon))dx.
\end{equation}

We need to mention some properties of the Cauchy transform. For a measure $\mu \in \mathcal{M}$ the Cauchy transform is analytic in $\mathbb{C}^+$ and maps $\mathbb{C}^+$ to $\mathbb{C}^-$ (and vice versa). If $supp(\mu)\subset[-K,K]$ , then
\begin{equation}\label{CTsupp}
G_{\mu}(x)\neq 0 \quad \text{for}\: x\in \mathbb{R}\setminus [-K,K].
\end{equation}
The Cauchy transform of the dilation $D_a\mu$ is given by 
 \begin{equation}\label{gtd}
G_{D_a\mu}(z)=\frac{1}{a}G_{\mu}(\frac{z}{a})\quad \text{for}\: z\in \mathbb{C}^+.
\end{equation}

Now, we give two more transforms before defining the Boolean convolution. The reciprocal Cauchy transform (or $F$-transform) of a measure $\mu \in \mathcal{M}$ is defined as
$$F_{\mu}(z):=\frac{1}{G_{\mu}(z)}\quad \:\text{for} z\in \mathbb{C}^+,$$
and the self energy (or $K$-transform) of a measure $\mu \in \mathcal{M}$ is defined as
$$K_{\mu}(z):=z-F_{\mu}(z)\quad \text{for}\: z\in \mathbb{C}^+.$$

\begin{remark}\label{rem.cauchy}
Since the Cauchy transform $G_{\mu}$ determines uniquely to the measure $\mu$ in $\mathbb{C}^+$, then also the $F$-transform $F_{\mu}$ and the $K$-transform $K_{\mu}$ do.
\end{remark}

Given $\mu,\nu \in \mathcal{M}$ the \emph{Boolean convolution} $\mu \uplus \nu \in \mathcal{M}$, introduced by Speicher and Woroudi \cite{SW}, is defined by the equation
$$K_{\mu \uplus \nu}(z)=K_{\mu}(z)+K_{\nu}(z) \quad \text{for} \: z\in \mathbb{C}^+.$$
Thus, to obtain $\mu \uplus \nu$, we calculate $G_{\mu \uplus \nu}(z)$ from $K_{\mu \uplus \nu}(z)$, and then we use the Stieltjes inversion formula.

\subsection{Properties of the $F$-transform}
We use the $F$-transform to identify potential atoms of a measure: if $a\in \mathbb{R}$ is an atom of $\mu \in \mathcal{M}$, then $F_{\mu}(a)=0$. The properties of this subsection are important for that purpose.

It is easy to see that if $\mu \in \mathcal{M}$ and $a>0$, then
\begin{equation}\label{ftd}
F_{D_a \mu}(z)=aF(z/a) \quad \text{for } z\in \mathbb{C}^+.
\end{equation}
Directly from the definitions above we obtain that  
\begin{equation}\label{ftc}
F_{\mu^{\uplus n}}(z)=(1-n)z-nF_{\mu}(z) \quad \text{for}\:z\in \mathbb{C}^+.
\end{equation}

The next proposition is a direct consequence of  Proposition 2.2 in Maassen \cite{Maa}.
\begin{prop}\label{neva}
Let $\mu$ be a probability measure on $\mathbb{R}$. Then $\mu \in \mathcal{M}_0^1$ if and only if there exists a measure $\nu \in \mathcal{M}$ such that
$$F_{\mu}(z)=z+\int_{-\infty}^{\infty} \frac{1}{t-z}d\nu(t)=z-G_{\nu}(z) \quad \text{for} \: z\in \mathbb{C}^+ .$$
\end{prop} 
The following result is Lemma 2.1 in Hasebe \cite{Has}.
\begin{prop}\label{hase}
Let $\mu$ and $\nu$ be probability measures on $\mathbb{R}$ such that $F_{\mu}(z)=z-G_{\nu}(z)$ for $z\in \mathbb{C}^+$. Then: 
\begin{itemize}

\item[1)]
$\mathbb{C} \setminus  supp(\mu)$ is the maximal domain where $G_{\mu}(z)$ is analytic.
\item[2)]
$\mathbb{C} \setminus  supp(\nu) $ is the maximal domain where $F_{\mu}(z)$ is analytic, and $F_{\mu}(z)=z-G_{\nu}(z)$ there.
\item[3)]
$\lbrace x\in  \mathbb{R} \setminus  supp(\mu)\:|\: G_{\mu}(x)\neq 0 \rbrace \subset \mathbb{R} \setminus  supp(\nu)$, and $\lbrace x\in \mathbb{R} \setminus  supp(\nu) \:|\: F_{\mu}(x)\neq 0 \rbrace \subset \mathbb{R} \setminus  supp(\mu).$
\end{itemize}
\end{prop}

\subsection{Boolean cumulants}
For a probability measure $\mu$ with all moments,  the Boolean cumulants are defined as 
the coefficients $r_n = r_n (\mu)$ in the series
\begin{equation}\label{Boolean cumulants}
K_\mu(z) =\sum^\infty_{n=1}r_nz^{1-n}.
\end{equation}

For a probability measure $\mu$ on $\mathbb{R}$ with finite fourth moment, we may also define the first four Boolean cumulants $r_1(\mu)$, $r_2(\mu)$, $r_3(\mu)$, and $r_4(\mu)$ of $\mu$ by the equations
\begin{eqnarray*}
m_1(\mu)&=&r_1(\mu),\\
m_2(\mu)&=&r_1(\mu)^2+r_2(\mu),\\
m_3(\mu)&=&r_1(\mu)^3+2r_1(\mu)r_2(\mu)+r_3(\mu),\: \text{and}\\   
m_4(\mu)&=&r_1(\mu)^4+3r_1(\mu)^2r_2(\mu)+r_2(\mu)^2+2r_3(\mu)r_1(\mu)+r_4(\mu),       
\end{eqnarray*}
where $m_n(\mu)$ is the $n$-th moment of $\mu$.
Note that for $\mu \in \mathcal{M}_0^1$ we have that
\begin{equation}\label{fourmoment.cumul}
m_4(\mu)=1+r_4(\mu).
\end{equation}

Similarly to the classical cumulants, the Boolean cumulants defined above satisfy for $\mu,\nu\in \mathcal{M}$ and $i\in\mathbb{N}$ that
\begin{equation}\label{cumu.adit}
r_i(\mu \uplus \nu)=r_i(\mu)+r_i(\nu)   
\end{equation}
and
\begin{equation}\label{cumu.dila}
r_i(D_a\mu)=a^ir_i(\mu).
\end{equation}

\section{Proofs}
We first prove a theorem that gives the L\'evy distance between a measure and the Bernoulli distribution in terms of the fourth moment. As a direct consequence, we obtain Theorem $\ref{teo.noacotado}$. Then we prove Theorem $\ref{teo.acotado}$ and give an example  that shows that the bound in this case is sharp.

Let $\mu$ be a probability measure, and let $X$ be a random variable with distribution $\mu$. By $\mu^2$ we denote the distribution of $X^2$. Note that for $\epsilon>0$ we have
\begin{equation}\label{mu.mu2} 
\mu((-1-\epsilon,-1+\epsilon)\cup (1-\epsilon,1+\epsilon))\geq \mu^2((1-\epsilon,1+\epsilon)).
\end{equation}

\begin{lemma}\label{cumul}
If $\mu \in \mathcal{M}_0^1$ and $m_4(\mu)<\infty$, then:

\begin{itemize}
\item[(i)]
$Var(\mu^2)=r_4(\mu)$.
\item[(ii)]
$\mu^2((1-\epsilon,1+\epsilon))\geq 1-\frac{r_4(\mu)}{\epsilon^2}$.

\end{itemize}
\end{lemma}

\begin{proof}[Proof]
For part \it{(i)} observe that  $Var(\mu^2)=E(\mu^4)-E(\mu^2)^2=m_4(\mu)-1.$
Hence, by (\ref{fourmoment.cumul}) we obtain that $Var(\mu^2)=r_4(\mu)$.

For part \it{(ii)} we see that by the Chebyshev inequality we have
$$P(|X^2-E(X^2)|<\epsilon)>1-\frac{Var(X^2)}{\epsilon^2},$$
and using (i) we conclude that
$\mu^2((1-\epsilon,1+\epsilon))>1-\frac{r_4(\mu)}{\epsilon^2}.$
\end{proof}

\begin{lemma}\label{ldmuab}
If $\mu \in \mathcal{M}_0^1$ and $\mu((-1-\epsilon,-1+\epsilon)\cup (1-\epsilon,1+\epsilon))\geq 1-\epsilon$ for some $\epsilon\in (0,1)$, then
\begin{equation}
L(\mu,\mathbf{b})\leq \frac{7}{2}\epsilon.
\end{equation}
\end{lemma}

\begin{proof} [Proof]
Define $R_1:=(-\infty,-1-\epsilon]$, $R_2:=(-1-\epsilon,-1+\epsilon)$, $R_3:=[-1+\epsilon,1-\epsilon]$, $R_4:=(1-\epsilon,1+\epsilon)$, and $R_5:=[1+\epsilon,\infty)$. Let $p_i:=\mu(R_i)$ for $i=1,2,3,4,5$. Clearly, there exists $t_i \in R_i $ such that
$$\int_{R_i}td\mu(t)=t_ip_i.$$
Note that  $p_1+p_4\geq 1-\epsilon$ by hypothesis. So we have 
\begin{equation}\label{equ.p135}
p_i\leq\epsilon \quad \text{for} \: i=1,3,5.
\end{equation}

Observe that
\begin{align}
|t_1p_1|+|t_5p_5|
&<\int_{R_1}t^2d\mu(t)+\int_{R_5}t^2d\mu(t) 
\nonumber \\
&=1-\int_{R_2}t^2d\mu(t)-\int_{R_3}t^2d\mu(t)-\int_{R_4}t^2d\mu(t)
\nonumber \\
&\leq 1-(1-\epsilon)^2(p_2+p_4)
\nonumber \\
&=1-(1-\epsilon)^2(1-\epsilon)
\nonumber \\
&<3\epsilon,
\nonumber
\end{align}
where the first equality is because of $m_2(\mu)=1$. It is clear that $|t_3p_3|<\epsilon$. Thus, we obtain from $m_1(\mu)=0$ that
$$|t_2p_2+t_4p_4|=|t_1p_1+t_3p_3+t_5p_5|<4\epsilon.$$
Also note that
\begin{align}
|p_2-p_4|-|t_2p_2+t_4p_4|
&\leq |t_2p_2+t_4p_4+p_2-p_4|
\nonumber \\
&=|t_2+1|p_2+|t_4-1|p_4
\nonumber \\
&<\epsilon(p_2+p_4)
\nonumber \\
&<\epsilon.
\nonumber
\end{align}
It follows that $|p_2-p_4|<5\epsilon$, and since $1-\epsilon \leq p_2+p_4 \leq 1$, then 
\begin{equation}\label{equ.p24}
\frac{1}{2}-3\epsilon<p_2,p_4<\frac{1}{2}+\frac{5}{2}\epsilon.
\end{equation}
Using the estimates ($\ref{equ.p135}$) and ($\ref{equ.p24}$), it is easy to see that
$$L(\mu,\mathbf{b})\leq \frac{7}{2}\epsilon.$$

\end{proof}

\begin{theo}\label{theo.levy.mu.b}
Let $\mu\in \mathcal{M}_0^1$. Then 
\begin{equation}
L(\mu,\mathbf{b})\leq \frac{7}{2}\sqrt[3]{m_4(\mu)-1}.
\end{equation}
\end{theo}

\begin{proof}[Proof]
By Lemma $\ref{cumul}$ and inequality ($\ref{mu.mu2}$), we see that
$$\mu((-1-\epsilon,-1+\epsilon)\cup (1-\epsilon,1+\epsilon))\geq \mu^2((1-\epsilon,1+\epsilon))\geq 1-\frac{r_4(\mu)}{\epsilon^2}.$$

Taking $\epsilon=\sqrt[3]{r_4(\mu)}$, we obtain
$$\mu( (-1-\sqrt[3]{r_4(\mu)},-1+\sqrt[3]{r_4(\mu)} )\cup (1-\sqrt[3]{r_4(\mu)},1+\sqrt[3]{r_4(\mu)}))\geq1-\sqrt[3]{r_4(\mu)}.$$
By Lemma $\ref{ldmuab}$ we conclude  that when $r_4(\mu)<1$, then
$$L(\mu,\mathbf{b})\leq \frac{7}{2}\sqrt[3]{r_4(\mu)}=\frac{7}{2}\sqrt[3]{m_4(\mu)-1}.$$
\end{proof}

The following proposition shows that the bound in the previous theorem is sharp.
\begin{prop}
For all $\alpha>\frac{1}{3}$ and for all $C>0$ there exists $\mu \in \mathcal{M}_0^1$ such that 
\begin{equation}\nonumber
L(\mu,B)>C\cdot r_4(\mu)^{\alpha}
\end{equation} 
\end{prop}
\begin{proof}
Fix $\alpha>\frac{1}{3}$ and $C>0$. Let $\epsilon\in (0,1)$. Define 
\begin{equation}\nonumber
\mu_{\epsilon}=\frac{\epsilon}{2}\delta_{-\sqrt{1+\epsilon}}+(\frac{1}{2}-\epsilon)\delta_{-1}+\frac{\epsilon}{2}\delta_{-\sqrt{1-\epsilon}}+\frac{\epsilon}{2}\delta_{\sqrt{1-\epsilon}}+(\frac{1}{2}-\epsilon)\delta_{1}+\frac{\epsilon}{2}\delta_{\sqrt{1+\epsilon}}
\end{equation}
Clearly $\mu_{\epsilon} \in \mathcal{M}_0^1$. We also have that
\begin{align}
m_4(\mu_{\epsilon})
&=\epsilon(1+\epsilon)^2+\epsilon(1-\epsilon)^2+(1-2\epsilon) \nonumber \\
&=1+2\epsilon^3. \nonumber
\end{align}
So by ($\ref{fourmoment.cumul}$) we obtain that $r_4(\mu)=2\epsilon^3$.

On the other hand, since $\mu_\epsilon$ is atomic, then,  for  $\epsilon<1$, one has 
\begin{equation}\nonumber
L(\mu_{\epsilon},B)=min\lbrace \frac{\epsilon}{2}, 1-\sqrt{1-\epsilon}, \sqrt{1+\epsilon}-1\rbrace\geq\frac{\epsilon}{4.}
\end{equation}

It is not hard to see that for $\epsilon$ small enough we have that $\frac{\epsilon}{4}>C(2^{\alpha}\epsilon^{3\alpha})$. For such $\epsilon$  
\begin{equation}\nonumber
L(\mu_{\epsilon},B)>C\cdot r_4(\mu)^{\alpha}.
\end{equation} 
\end{proof}

Now we are able to prove Theorem \ref{teo.noacotado} .

\begin{proof}[Proof of Theorem \ref{teo.noacotado}]
It follows from (\ref{cumu.adit}) and (\ref{cumu.dila}) that 
$$r_4(\mu_n)=r_4(D_{\sqrt{n}}\mu^{\uplus n})=\frac{1}{n^2}n r_4(\mu)=\frac{r_4(\mu)}{n}.$$

By Theorem $\ref{theo.levy.mu.b}$ we conclude that
$$L(\mu_n,\mathbf{b})\leq \frac{7}{2}\sqrt[3]{\frac{r_4(\mu)}{n}}=\frac{7}{2}\sqrt[3]{\frac{m_4(\mu)-1}{n}}.$$
\end{proof}

Now we proceed to prove the main theorem of the paper.

\begin{proof}[Proof of Theorem 2]
Using ($\ref{ftd}$) and ($\ref{ftc}$), we obtain
$$F_{\mu_n}(z)=(1-n)z-\sqrt{n}F_{\mu}(\sqrt{n}z) \quad \text{for} \:z\in\mathbb{C}^+.$$
By Proposition $\ref{neva}$ there exists a measure $\nu\in \mathcal{M}$ such that $F_{\mu}(z)=z-G_{\nu}(z)$ for $z\in \mathbb{C}^+$.
It follows that
\begin{align*}
F_{\mu_n}(z)&=(1-n)z-\sqrt{n}(\sqrt{n}z-G_{\nu}(\sqrt{n}z)) \nonumber\\
            &=z-\sqrt{n}G_{\nu}(\sqrt{n}z))\nonumber \\
            &=z-G_{D_{\frac{1}{\sqrt{n}}}\nu}(z) \qquad for \: z\in \mathbb{C}^+,
\end{align*}
where the second equality is due to (\ref{gtd}).

 Note that $supp \: (\nu) \subset [-K,K]$. Indeed, suppose that $x\in \mathbb{R}\setminus [K,K]$. Since $supp \: (\mu) \subset [-K,K]$, then $x\in \mathbb{R}\setminus supp \: (\mu)$. Therefore, by ($\ref{CTsupp}$) we obtain that $G_{\mu}(x)\neq 0$. Finally, part iii) of Proposition $(\ref{hase})$ implies that  $x\in \mathbb{R}\setminus supp \: (\nu)$.

Let us write $\hat{\nu}=D_{\frac{1}{\sqrt{n}}}\nu$. So we  have that $supp (\hat{\nu}) \subset [\frac{-K}{\sqrt{n}},\frac{K}{\sqrt{n}}]$ and
\begin{equation} \label{u_n,v_n}
F_{\mu_n}(z)=z-G_{\hat{\nu}}(z) \quad \text{for}\: z\in \mathbb{C}^+. 
\end{equation} 

By the third part of Proposition \ref{hase}, we conclude that
$supp(\mu_n) \subset [\frac{-K}{\sqrt{n}},\frac{K}{\sqrt{n}}] \cup \lbrace x\in \mathbb{R} \setminus  [\frac{-K}{\sqrt{n}},\frac{K}{\sqrt{n}}] \:|\: F_{\mu_n}(x)=0 \rbrace$. To conclude the proof of part 1), it is left to prove that there are only two zeros $x_1$ and $x_2$ which satisfy the conditions $|(-1)-x_1|\leq \frac{K}{\sqrt{n}}$ and $|1-x_2|\leq \frac{K}{\sqrt{n}}$. The second part of Proposition $\ref{hase}$ implies for $z\in \mathbb{C}\setminus [\frac{-K}{\sqrt{n}},\frac{K}{\sqrt{n}}]$ that  $F_{\mu_n}(z)=z-G_{\hat{\nu}}(z)$ and that $F_{\mu_n}(z)$ is analitic. Therefore, we obtain from the definition of Cauchy transform that 
  
$$F_{\mu_n}'(x)=1+\int_{-\infty}^{\infty} \frac{1}{(t-x)^2}d\hat{\nu}(t) \quad \text{for}\: x\in \mathbb{R}\setminus [\frac{-K}{\sqrt{n}},\frac{K}{\sqrt{n}}].$$
In particular, $F_{\mu_n}(x)$ is increasing in $(\frac{K}{\sqrt{n}},\infty)$ and can have at most one zero there.
It is clear that
$$F_{\mu_n}(x)>x-\frac{1}{x-K/\sqrt{n}}>K/\sqrt{n} \quad \text{for}\: x>1+K/\sqrt{n}$$
and 
$$F_{\mu_n}(x)<x-\frac{1}{x+K/\sqrt{n}}<-K/\sqrt{n} \quad \text{for}\: K/\sqrt{n}<x<1-K/\sqrt{n}.$$
Since $F_{\mu_n}(x)$ is continous in $(\frac{K}{\sqrt{n}},\infty)$, it must have a zero $x_2$ in $[1-\frac{K}{\sqrt{n}},1+\frac{K}{\sqrt{n}}]$.

A similar argument shows that $F_{\mu_n}(x)$ has only a zero $x_1$ in $(-\infty,\frac{-K}{\sqrt{n}})$ bounded in $[-1-\frac{K}{\sqrt{n}},-1+\frac{K}{\sqrt{n}}]$. We conclude the proof of part 1).

Using (\ref{u_n,v_n}) and Proposition \ref{neva}, we obtain that $m_1(\mu_n)=0$ and $m_2(\mu_n)=1$. The idea of the rest of the proof is that these two moments force the mass of $\mu_n$ to concentrate evenly in $x_1$ and $x_2$.

We put $p=\mu_n(\lbrace x_1 \rbrace)$, $q=\mu_n(\lbrace x_2 \rbrace)$, and $r=\mu_n([\frac{-K}{\sqrt{n}},\frac{K}{\sqrt{n}}])$. Note that $p+q+r=1$ by part 1). It is clear that there exist  $y_1 \in [\frac{-K}{\sqrt{n}},\frac{K}{\sqrt{n}}]$ and $y_2 \in [0,\frac{K}{\sqrt{n}}]$ such that 
$\int_{-K/\sqrt{n}}^{K/\sqrt{n}} x \: d\mu_n(x)=y_1r$ and $\int_{-K/\sqrt{n}}^{K/\sqrt{n}} x^2 \: d\mu_n(x)=y_2^2r$. Define $\epsilon:=-1-x_1$ and $\delta:=1-x_2$. Since $m_1(\mu_n)=0$, then we have that $x_1p+y_1r+x_2q=0$; it follows that $p(-1+\epsilon)+ry_1+q(1+\delta)=0$. Therefore, we deduce the inequalities

\begin{equation}\label{pq.ine}
|q-p|\leq p|\epsilon|+r|y_1|+q|\delta|\leq (p+q+r)\frac{K}{\sqrt{n}}=\frac{K}{\sqrt{n}}.
\end{equation}

Since $p+q\leq 1$, then $p+p-\frac{K}{\sqrt{n}}\leq 1$. Therefore we obtain that $p\leq \frac{1}{2}+\frac{K}{2\sqrt{n}}$. Similarly, we can conclude that $q\leq \frac{1}{2}+\frac{K}{2\sqrt{n}}$. Now, we have from $m_2(\mu_n)=1$
that $x_1^2p+y_2^2r+x_2^2q=1$. It follows that $(1+2\epsilon+\epsilon^2)p+y_2^2r+(1+2\delta+\delta^2)q=1$, and we get the estimate
$$p+q=1-(\epsilon^2p+y_2^2r+\delta^2q)-2(\epsilon p + \delta q)\ge 1-\frac{K^2}{n}-2\frac{K}{\sqrt{n}}\ge 1-3\frac{K}{\sqrt{n}}.$$
Since $q\leq p + \frac{K}{\sqrt{n}}$, then $2p+\frac{K}{\sqrt{n}}\ge 1-3\frac{K}{\sqrt{n}}$. It follows that $p\ge \frac{1}{2}-\frac{2K}{\sqrt{n}}$ and $q\ge \frac{1}{2}-\frac{2K}{\sqrt{n}}$. Finally, we conclude that $r\leq 4\frac{K}{\sqrt{n}}$ because of $p+q+r=1$.

It follows from the estimates obtained for $p, q$, and $r$ that
$$L(\mu_n,\mathbf{b})\leq \frac{2K}{\sqrt{n}}.$$

\end{proof}

Finally, the next example shows that the estimate given for measures of bounded support is sharp.

\begin{exa}\label{ejemplo} 
Let $n$ be a positive integer. Define $p_n:=\frac{1}{2}\frac{\sqrt{1+4n}+1}{\sqrt{1+4n}}$, $q_n:=\frac{1}{2}\frac{\sqrt{1+4n}-1}{\sqrt{1+4n}}$, $x_n:=\frac{1-\sqrt{1+4n}}{\sqrt{4n}}$, and $y_n:=\frac{1+\sqrt{1+4n}}{\sqrt{4n}}$. Let $\mu_n$ be the probability measure given by $\mu_n:=p_n\delta_{x_n}+q_n\delta_{y_n}$.  Then $\mu_n=D_{\frac{1}{\sqrt{n}}}\mu_1^{\uplus n}$, where $\mu:=\mu_1$, and $L(\mu_n,\mathbf{b})\ge \frac{1}{6\sqrt{n}}$.

 Indeed, by Remark $\ref{rem.cauchy}$, to prove $\mu_n=D_{\frac{1}{\sqrt{n}}}\mu^{\uplus n}$, it is sufficient to show that $F_{\mu_n}(z)=F_{D_{\frac{1}{\sqrt{n}}}\mu^{\uplus n}}(z)$ for $z\in \mathbb{C}^+.$
First we compute the Cauchy transform of $\mu_n$:

\begin{align}
G_{\mu_n}(z)&=\frac{p_n}{z-x_n}+\frac{q_n}{z-y_n}=\frac{p_n(z-y_n)+q_n(z-x_n)}{(z-x_n)(z-y_n)} \nonumber\\
&=\frac{p_n(z-y_n)+q_n(z-x_n)}{(z-x_n)(z-y_n)}=\frac{z-y_np_n-x_nq_n}{(z-x_n)(z-y_n)}\nonumber \\
&=\frac{z-x_n-y_n}{z^2-(x_n+y_n)+x_ny_n}=\frac{z-1/\sqrt{n}}{z^2-z/\sqrt{n}-1}.\nonumber   
\end{align}
Hence, we have that $F_{\mu_n}(z)=\frac{\sqrt{n}z-z-\sqrt{n}}{\sqrt{n}z-1}$. In particular $F_{\mu}(z)=\frac{z^2-z-1}{z-1}$.

On the other hand, we compute 
\begin{align}
F_{D_{\frac{1}{\sqrt{n}}}\mu^{\uplus n}}(z)&=\frac{1}{\sqrt{n}}((1-n)\sqrt{n}z+nF_{\mu}(\sqrt{n}z)=(1-n)z+\sqrt{n}F_{\mu}(\sqrt{n}z) \nonumber\\
&=(1-n)z+\sqrt{n}(\frac{nz^2-\sqrt{n}z-1}{\sqrt{n}z-1})=(z-nz)(\sqrt{n}z-1)+\sqrt{n}nz^2-nz-\sqrt{n},\nonumber \\
&=\frac{\sqrt{n}z-z-\sqrt{n}}{\sqrt{n}z-1}=F_{\mu_n}(z).\nonumber  
\end{align}
Now it is easy to see, that $L(\mu_n,\mathbf{b})=max\{|-1-x_n|,|1-y|,|1/2-p_n|\}$, and since $|1/2-p_n|\ge \frac{1}{6\sqrt{n}}$, then $L(\mu_n,\mathbf{b})\ge \frac{1}{6\sqrt{n}}$.

\end{exa}



\begin{thebibliography}{100}



\bibitem{Ber} Berry, A. C. (1941). The accuracy of the gaussian approximation to the
sum of independent variates. \emph{Transactions of the American Mathematical
Society}, 49(1):122--136.
\bibitem{CG} Chistyakov, G. P. and G{\"o}tze, F. (2008). Limit theorems in free probability
theory. i.  \emph{The Annals of Probability}, pages 54--90.
\bibitem{Es} Esseen, C.-G. (1942).  \emph{On the Liapounoff limit of error in the theory of
probability.} Almqvist \& Wiksell.
\bibitem{Has} Hasebe, T. (2010). Monotone convolution semigroups. \emph{Studia Math},
200(2):175--199.
\bibitem{Maa} Maassen, H. (1992). Addition of freely independent random variables.
\emph{Journal of Functional Analysis}, 106(2):409--438.
\bibitem{Mur} Muraki, N. (2001). Monotonic independence, monotonic central limit theorem
and monotonic law of small numbers. \emph{Infinite Dimensional Analysis,
Quantum Probability and Related Topics}, 4(01):39--58.
\bibitem{Mur2} Muraki, N. (2003). The five independences as natural products. \emph{Infinite Dimensional Analysis, Quantum Probability and Related Topics},
6(03):337--371.
\bibitem{SW} Speicher, R. and Woroudi, R. (1997). Boolean convolution. \emph{Fields Inst.
Commun}, 12:267--279.
\bibitem{Voi} Voiculescu, D. (1985). Symmetries of some reduced free product $C^*$
algebras, in operator algebras and their connections with topology and
ergodic theory pp. 556--588. \emph{Lecture Notes in Mathematics}, 1132.
\end{thebibliography}
\end{document}